%% file: main.tex
\documentclass{article}

\input{base}

\input{logic}

\newcounter{dummy} \numberwithin{dummy}{section}
\newtheorem{Theorem}[dummy]{Theorem}
\newtheorem{Definition}[dummy]{Definition}
\newtheorem{Lemma}[dummy]{Lemma}
\newtheorem{Corollary}[dummy]{Corollary}

\setlist{noitemsep, topsep=5pt}
\usepackage{setspace} 

\title{Set-Theoretic Hypodoxes and co-Russell's Paradox}
\author{Timotej Šujan\\\\Charles University, Prague}
\date{}

\begin{document}

\maketitle
\begin{abstract}
    In this paper, we argue that while the concept of a set-theoretic paradox (or paradoxical set) can be relatively well-defined within a formal setting, the concept of a set-theoretic hypodox (or hypodoxical set) remains significantly less clear—especially if the self-membership assertion of the co-Russell set, $\{x:x\in x\}$, is classified as hypodoxical, whereas other set-theoretic sentences with no apparent connection to paradoxes are not. Furthermore, we demonstrate in detail how a contradiction can be derived in Na\"{\i}ve Set Theory by exploiting the unique properties of the co-Russell set, relying on the Fixed Point Theorem of Na\"{\i}ve Set Theory. This result suggests that the boundary between paradoxes and hypodoxes may not be as clear-cut as one might assume.
\end{abstract}
\begin{spacing}{0.5}
    {\tableofcontents}
\end{spacing}

\paragraph{Acknowledgements} This work was supported by grant SVV 260677/2024. The author also expresses gratitude to the \emph{Mathematics Stack Exchange} community for their helpful responses.

\section{Introduction}
In formal logic, we have had considerable success with paradoxical and self-referential constructs when used as tools. Famous examples include the results in \citep{godel}, \citep{lob} and \citep{solovay}. In this paper, we explore a particular kind of self-referential object in the hope that it can be fruitfully used as a tool in a formal context.

In recent years, several philosophers have investigated self-referential constructs that are considered duals of paradoxes (paradoxical entities), referring to them as hypodoxes\footnote{The term "hypodox" was likely coined in \citep{smith2007}.} (hypodoxical entities). These hypodoxical entities often\footnote{See \cite[p. 2465]{smith2022} for an argument that sometimes simply negating key properties is not enough, using the \textit{Knower paradox} as an example.} arise by negating key properties of paradoxical ones. The most frequently cited\footnote{See, e.g., \cite[p. 298]{mackie}, \citep{billon}, or \citep{smith2023}.} examples of hypodoxical entities are the co-Russell set $\{x:x\in x\}$ and the Truth-teller sentence, "This sentence is true."

Does the co-Russell set belong to itself? Is the Truth-teller sentence true? Hypodox researchers argue that these questions can be answered either way and that it is unclear why a positive answer should be preferred over a negative one, or vice versa.

Current research on the concept of a hypodox is still in the process of establishing a suitable definition. A recent definition is as follows:

\begin{displayquote}
    Something, X, is a \textit{hypodox} iff it is underdetermined by granted principles, in given circumstances in some cases, and these circumstances
(if any) and principles are consistent with adding a principle so that it
would result that X is the case and consistent with instead adding another principle so that it would result that X is not the case; but no such
principle has been granted, usually because it has not been justified. \cite[p. 6]{smith2024}
\end{displayquote}
This definition is rather vague, but it is evident that the concept of hypodox is closely tied to that of independence. However, under this definition, it remains unclear why any formal sentence that is independent in some formal system would not qualify as a~hypodox. This raises the concern that the concept of a~hypodox might be redundant, at least in formal contexts.

On the other hand, some statements by hypodox researchers suggest that hypodoxical sentences exhibit, in some sense, a stronger or more profound form of independence than other set-theoretic sentences. For instance, \cite[p. 2]{billon} cautiously remarks about the co-Russell set: “\dots it is hard to see what could make it the case that it is, or that it is not, itself self-membered\dots.” Similarly, \cite[p. 1]{smith2023} states explicitly: “\dots[the co-Russell set] is a member of itself xor it is not, but no principle of classical logic or naive set theory determines which.”

If this were indeed the case—if hypodoxes were in some sense more independent than other formal sentences—then the concept of hypodox would gain a unique significance in formal contexts.

In this paper, we argue that while the concept of a set-theoretic paradox (or paradoxical set) can be relatively well-defined in a formal setting, the concept of a~set-theoretic hypodox (or hypodoxical set) remains unclear—particularly if the self-membership assertion of the co-Russell set is classified as hypodoxical, whereas other set-theoretic sentences with no evident connection to paradoxes are not.

We focus on the co-Russell set—the primary example of a hypodoxical set\footnote{To our surprise, we found no mention in the literature of duals of other paradoxical sets, such as the set of all well-founded sets or the set of all ordinals. We suspect this is because it is immediately clear that the self-membership of the duals of these paradoxical sets—namely, the set of all non-well-founded sets and the set of all non-ordinals—is not independent even under very weak assumptions.}—and examine its self-membership assertion. Our investigation aims to determine whether any unique property tied to its independence distinguishes this assertion from other set-theoretic statements. In light of L\"{o}b’s results \citep{lob}, we remain cautious about prematurely judging the independence of self-referential sentences.

Finally, we demonstrate how a contradiction can be derived in Na\"{\i}ve Set Theory using the unique properties of the co-Russell set, relying on the Fixed Point Theorem of Na\"{\i}ve Set Theory.

\section{Framework and set-theoretic paradoxes}

Let us define \textit{basic set theory} ($\msf{BST}$) as classical predicate logic with equality, along with the axiom of extensionality ($\msf{exten}$):
\[
\begin{aligned} 
    \text{axioms of equality:}&\quad\forall a (a = a) \quad \mid \quad \forall a \forall b ((a = b) \rightarrow (\varphi(a) \rightarrow \varphi(b)))  \\
    \msf{exten}{\textup{:}}&\quad\forall y \forall z (\forall x (x \in y \leftrightarrow x \in z) \rightarrow y = z)  
\end{aligned}
\]

We propose the following definition of a set-theoretic paradox: a set-theoretic paradox is a proof of a contradiction in a seemingly consistent extension of $\msf{BST}$. By a \textit{seemingly consistent} system, we mean a system that was strongly believed to be consistent but was later shown to be inconsistent. Determining what qualifies as a seemingly consistent system is, of course, a difficult philosophical question. Nevertheless, we have at least one\footnote{As we will show later, there is a sense in which considering only $\msf{NST}$ as a seemingly consistent system suffices, since virtually any other set theory extending $\msf{BST}$ can be viewed as a fragment of $\msf{NST}$. Consequently, any paradox arising in a seemingly consistent system can also be reconstructed in the corresponding fragment of $\msf{NST}$ by first deriving the relevant axioms within it and then deriving the contradiction.} prominent example: Na\"{\i}ve Set Theory ($\msf{NST}$), which was widely considered consistent but ultimately proved otherwise.

Let us formalize $\msf{NST}$ as $\msf{BST}$ with the following axiom schema of unrestricted comprehension ($\msf{UC}$):
\[
    \msf{UC}{\textup{:}}\quad\exists y\forall x(x\in y \leftrightarrow \varphi(x)),
\]
where $y$ is not free in $\varphi$.

We call an $\msf{NST}$-fragment any extension of $\msf{BST}$ by a set $\Gamma$ of instances of $\msf{UC}$, and we refer to such an extension as $\msf{BST}{+}\Gamma$. We adopt similar notation for extension of $\msf{BST}$ by other sentences (not necessarily $\msf{UC}$-instances), such as $\alpha$ and $\beta$, writing such extension as $\msf{BST}{+}\alpha{+}\beta$ or $\msf{BST}{+}\{\alpha,\beta\}$.

We use a set-builder notation, e.g. $\{x:\varphi(x)\}$, only informally or as an abbreviation for instances of $\msf{UC}$, such as $\forall x(x\in y \lra \varphi(x))$. This usage is purely for convenience, and we do not intend to suggest that set-builder notation is part of the formal language employed.

For any given $\msf{UC}$-instance, we call the set whose existence is asserted through the variable bound by the first existential quantifier (in the schema above, this is the variable $y$) the \textit{main set} of that $\msf{UC}$-instance.

\begin{Definition}[paradoxical sentence/set]
    A $\msf{UC}$-instance $\alpha$ is \emph{paradoxical} if $\mathsf{BST}{+}\alpha$ is a trivial theory \textup{(}i.e., a theory in which every sentence is provable\textup{)}. A~set $A$ is \emph{paradoxical} if it is the \textit{main set} of a paradoxical $\msf{UC}$-instance.
\end{Definition}

\begin{Theorem}\label{thmpara}
    The following sets are paradoxical\textup{:}
    \begin{itemize}
        \item[\textup{(a)}] Russell set $\{x:x\notin x\}$
        \item[\textup{(b)}] Curry-like sets $\{x:x\in x\rightarrow \varphi\}$ \textup{(}if $\varphi$ cannot hold in a given context\textup{)}
    \end{itemize}
\end{Theorem}
\begin{proof}
    (a) Let $R=\{x:\notin x\}$. The derivation is straightforward. Assuming $R\in R$ leads to $R\notin R$, and vice versa, yielding a contradiction.

    (b) Let $C=\{x:x\in x\ra \varphi\}$. Over $\msf{BST}$, $x\in x\rightarrow \varphi$ is equivalent to $x\notin x \lor \varphi$. Suppose $\neg\varphi$ holds. By extensionality, $C$ would then equal the Russell set $R$, leading to a contradiction. Hence, $\varphi$ must hold, making $C$ equal to the universal set $U=\{x:x{=}x\}$. Now, consider $\varphi:=\neg\exists z(z\in y \leftrightarrow z{=}z)$, which asserts that a universal set does not exist. This creates a contradiction: $U$ cannot exist, but $C=U$.
\end{proof}

We now present a somewhat counterintuitive result: the Mirimanoff set of all well-founded sets is not paradoxical under our definition.

A set $x$ is called \emph{transitive} if every element of $x$ is also a subset of $x$. Formally,
\[
    \msf{transitive}(x):= \forall y\forall z(((y\in z)\land(z\in x))\ra y\in x).
\]

A set $x$ is said to be \emph{well-ordered} if its elements are \emph{irreflexive}, \emph{transitive}, \emph{linear}, and \emph{well-founded} with respect to the relation $\in$. Formally,
\begin{align*}
    \msf{well{\text{-}}ordered}(x) := 
    &\quad\forall a(a\in x \ra a\notin a)\quad\land\\ 
    &\quad\forall a\forall b \forall c ((a\in x \land b\in x \land c\in x)\ra((a\in b \land b\in c)\ra a \in c))\quad\land\\
    &\quad\forall a \forall b((a\in x \land b \in x)\ra(a\neq b\ra (a\in b \lor b\in a)))\quad\land\\
    &\quad \forall s(s\subseteq x\ra(s{\neq}\emptyset\ra \exists m\forall z((m\in s \land z \in s)\ra z\notin m))).
\end{align*}

A set $x$ is called \emph{well-founded} if there exists a transitive, well-ordered set $y$ such that $x \subseteq y$. Formally,
\[
    \msf{well{\text{-}}founded}(x) := \exists y(\msf{transitive}(y)\land \msf{well{\text{-}}ordered}(y) \land x\subseteq y).
\]

\begin{Theorem}
    The Mirimanoff set, $\{x:\msf{well{\text{-}}founded}(x)\}$, is \textbf{not} paradoxical.
\end{Theorem}
\begin{proof}
    Let $M=\{x:\msf{well{\text{-}}founded}(x)\}$. Consider a model consisting of the sets $\emptyset$, $\{\emptyset\}$, $\{\emptyset,\{\emptyset\}\}$ and $\{\{\emptyset\}\}$. We expand this model by including $M$. 
    
    Observe that $\emptyset$, $\{\emptyset\}$ and $\{\emptyset,\{\emptyset\}\}$ are transitive and well-ordered sets. Hence, by definition, they are well-founded and belong to $M$.
    Additionally, $\{\{\emptyset\}\}$ is also well-founded, as its transitive closure, $\{\emptyset,\{\emptyset\}\}$, is well-ordered. Thus, $\{\{\emptyset\}\}\in M$.

    Now, consider whether $M$ is itself well-founded. This would require 
    $M$ to have a~transitive closure that is well-ordered by $\in$. The only candidate for such a closure is $M$ itself. However, $M$ is not well-ordered by $\in$, as $\emptyset\notin \{\{\emptyset\}\}$ and $\{\{\emptyset\}\}\notin\emptyset$, which violates linearity.
    Thus, $\msf{well{\text{-}}founded}(M)$ cannot hold, which implies $M\notin M$ in the given model.

    Since $\msf{BST}{+}\text{"there exists the Mirimanoff set"}$ has a model, it is a non-trivial theory. Therefore, the Mirimanoff set is not paradoxical.
\end{proof}

The contradiction in the previous proof did not arise merely because, in 
$\msf{BST}$-models, a transitive closure need not exist for every set.  Informally, within such models, where transitive closures are not guaranteed, not every set can be "truly" recognized as well-founded.

To capture the paradoxicality of the Mirimanoff set, we now extend the definition of a paradoxical $\msf{UC}$-instance to include cases where a contradiction arises only by assuming multiple $\msf{UC}$-instances.

\begin{Definition}
    The collection of $\msf{UC}$-instances $\Gamma$ is called a \textit{paradoxical group} if $\mathsf{BST}{+}\Gamma$ is a trivial theory, but for every proper subset $\Delta\subset \Gamma$, $\mathsf{BST}{+}\Delta$ is a~non-trivial theory.
\end{Definition}

Thus, while the Mirimanoff set may not be paradoxical on its own, it certainly belongs to a paradoxical group.

\begin{Theorem}
    For every natural number $n{>}0$, there exists a paradoxical group of size $n$.
\end{Theorem}
\begin{proof}
    If $n=1$, we take some paradoxical $\msf{UC}$-instance. Assume $n{>}1$.

    Let $\alpha_0,\alpha_1, \alpha_2,\dots$ be $\msf{UC}$-instances. Also, let $\alpha_k(x)$ be the formula derived from $\alpha_k$ by removing the initial existential quantifier, replacing variable $x$ with a new variable and replacing variable $y$ with $x$. If $\alpha_k := \exists y \forall x(x\in y \lra \varphi(x))$, then $\alpha_k(x)$ could be, e.g., $\forall z(z\in x\lra \varphi(z))$. Define $\alpha_0,\dots,\alpha_{n{-}2}$ to assert the existence of the ordinals $0,\dots,n{-}2$ as follows: 
    \begin{align*}
        \alpha_0 &:= \exists y\forall x(x\in y \lra x\neq x)\\
        \alpha_1 &:= \exists y\forall x (x\in y \lra(\alpha_0 \ra \alpha_0(x)))\\
        \alpha_2 &:= \exists y \forall x(x\in y \lra((\alpha_0\land\alpha_1) \ra(\alpha_0(x)\lor \alpha_1(x))))\\
        &\vdots\\
        \alpha_{n{-}2}&:= \exists y\forall x(x\in y \lra((\alpha_0\land\dots\land\alpha_{n{-}3})\ra(\alpha_0(x)\lor\dots\lor\alpha_{n{-}3}(x))))
    \end{align*}
    
    If all ordinals less than $k$ exist, then $\alpha_k$ asserts the existence of the ordinal $k$; otherwise, it asserts the existence of the universal set.

    Now, consider the $\msf{UC}$-instance $\beta$:
    \[
    \beta := \exists y\forall x(x\in y \lra ((\alpha_0\land\dots\land\alpha_{n{-}2})\ra x\notin x))
    \]
    If $\alpha_k$ holds for all $k$ such that $0\leq k\leq n{-}2$, then $\beta$ asserts the existence of the Russell set, yielding a contradiction. However, if some $\alpha_k$ does not hold, all $\alpha_l$ for $l>k$ and $\beta$ merely assert the existence of the universal set, which is unproblematic in this context. Thus, the collection $\{\alpha_0, \dots, \alpha_{n-2}, \beta\}$ forms a~paradoxical group of size $n$.
\end{proof}

\begin{Theorem}
    There does not exist a paradoxical group of infinite size.
\end{Theorem}
\begin{proof}
    Assume for contradiction that there exists a paradoxical group of infinite size. Let $\Gamma$ be such a group, where $\mathsf{BST{+}} \Gamma$ is a trivial theory, but for every proper subset $\Delta \subset \Gamma$, $\mathsf{BST{+}} \Delta$ is non-trivial. Any proof of a contradiction in $\mathsf{BST{+}} \Gamma$ depends on only a finite subset of $\Gamma$, say $\Delta_0 \subset \Gamma$. Since $\mathsf{BST{+}}\Delta_0$ is trivial, the full group $\Gamma$ cannot be paradoxical.
\end{proof}

Now we have seen that the concept of a set-theoretic paradox and a paradoxical set can be defined fairly clearly.

\section{Set-theoretic hypodoxes?}

In the previous section, we claimed that a fitting definition of a~set-theoretic paradox is a proof of a contradiction in a~seemingly consistent extension of $\msf{BST}$. We think that a fitting dual concept\footnote{A somewhat weaker dual concept could be a proof of the independence of the sentence $\varphi$ in a~seemingly $\varphi$-complete extension of $\msf{BST}$. By a seemingly $\varphi$-complete system, we mean one in which there was a strong belief that $\varphi$ is not independent, but which ultimately proves $\varphi$.} would be a proof of incompleteness in a seemingly \textit{complete} extension of $\msf{BST}$. By a \textit{seemingly complete} system, we mean one in which there was a strong belief\footnote{Of course, it may be controversial to claim that a given extension of $\msf{BST}$ is or is not seemingly complete. However, it seems less controversial to argue that, prior to Gödel's results, there was a~belief that at least some "nicely" axiomatizable (i.e., recursively axiomatizable) extension of $\msf{BST}$ was both complete and suitable as a foundation for mathematics. Thus, the independence proofs of Gödel's sentences in such a seemingly complete system would fit this definition of the dual concept to a paradox.} that no independent sentences exist, but which ultimately proves to be incomplete.

The problem with this dual definition is that it is not at all clear why the co-Russell set should be connected to it. To establish such a connection, we would need to identify an extension of $\msf{BST}$ in which the self-membership of the co-Russell set is independent, and argue that this extension is seemingly complete.

Another issue is that, even if we succeeded in finding such an extension, it would likely include many other independent sentences that fit the above definition but have never been considered hypodoxical.

Therefore, we take a different approach. Specifically, we will briefly investigate consistent extensions of $\msf{BST}$ to search for a property, tied to independence, that could distinguish the self-membership assertion of co-Russell set from other set-theoretic sentences. If such a property is found, we could use it to characterize set-theoretic hypodoxes and argue that the notion of hypodox is not redundant within the formal context.
As mentioned in the introduction, it appears that hypodox researchers assume such a property: namely, that hypodoxical sentences exhibit, in some sense, a stronger form of independence than other set-theoretic sentences.

First, we will show that there is a sense in which it suffices to investigate only consistent $\msf{NST}$-fragments, since virtually any set theory can be re-axiomatized (in a~very straightforward manner) as a consistent $\msf{NST}$-fragment.

\begin{Theorem}
    Any $\mathsf{BST}$-extension $\msf{BST{+}}\Gamma$ defined by a collection $\Gamma$ of axioms \textup{(}not necessarily $\msf{UC}$-instances\textup{)}, where at least one $\msf{UC}$-instance is provable, is equivalent \textup{(}in terms of provable sentences\textup{)} to some $\msf{NST}$-fragment.
\end{Theorem}
\begin{proof}
    Let $\exists y\forall x(x\in y \lra \varphi(x))$ be a provable $\msf{UC}$-instance in $\msf{BST{+}}\Gamma$. For each axiom $\psi\in\Gamma$, consider the following instance of $\msf{UC}$:
    \[
    \exists y\forall x(x\in y \lra ((\psi\ra\varphi(x))\land (\neg\psi\ra x\notin x)))
    \]
    This instance defines a more elaborate version of a Curry-like set. If $\neg\psi$ holds, the instance would yield a set equivalent to the Russell set, leading to a contradiction. Hence, $\neg\psi$ cannot hold. If $\psi$ holds, the instance simplifies to the already provable $\exists y\forall x(x\in y \lra \varphi(x))$. Thus, each such instance forces $\psi$ to hold while remaining equivalent to a set defined by $\varphi$.
\end{proof}

Now, we are ready to investigate the co-Russell set.

First, let us demonstrate that there exist extensions of $\msf{BST}$ in which the self-membership of the co-Russell set is independent.

\begin{Theorem}
    The self-membership of the co-Russell set is independent in $\msf{BST}$.
\end{Theorem}
\begin{proof}
    There exists a model in which the co-Russell set contains itself—namely, a~model with a single Quine atom. Conversely, there also exists a model in which the co-Russell set does not contain itself—specifically, a model consisting only of the empty set.
\end{proof}

While this is a promising start, it is far from sufficient to conclude that the co-Russell set has any uniquely significant independence property. This is because $\msf{BST}$ is an extremely weak theory with a vast number of independent statements. Moreover, $\msf{BST}$ contains many independent statements regarding the self-membership of various sets, none of which have ever been considered hypodoxical.

\begin{Theorem}
    The self-membership of the following sets is independent in $\msf{BST}{:}$\footnote{This theorem could potentially be extended to include certain notions of finite and infinite size, such as the set of all Dedekind-finite sets.}
    \begin{itemize}
        \item[\textup{(a)}] The set of all sets of size $n{>}0$, $n \in \N{:}$ $\{x : |x| = n\}$
        \item[\textup{(b)}] The set of all sets of size different from $n \in \N{:}$ $\{x : |x| \neq n\}$
    \end{itemize}
\end{Theorem}
\begin{proof}
    (a) Consider a model with $n{+}1$ Quine atoms.
    Expand this model by adding $n{-}1$ distinct sets, each containing exactly $n$ Quine atoms. Now, introduce the set $\{x : |x| = n\}$, which includes at least $n{-}1$ sets of size $n$. This set may or may not contain itself.

    (b) For $n{=}0$, consider two single-element models: one consisting of just the empty set, and another consisting of a single Quine atom.

    For $n{=}1$, start with a model containing only the empty set. Expand this model by introducing the set $\{x : |x| \neq 1\}$. This set may or may not contain itself.

    For $n{>}1$, consider a model with $n$ Quine atoms. If expanded by the set $\{x : |x| \neq n\}$, this set will have size $n$ and therefore cannot contain itself. Alternatively, consider a model with an infinite number of Quine atoms. If expanded by $\{x : |x| \neq n\}$, the resulting set must contain itself.
\end{proof}

So, we should move to stronger extensions. Let us investigate whether there exists a finitely axiomatizable consistent extension of $\msf{BST}$ in which the self-membership of the co-Russell set is no longer independent. It turns out that such an extension does indeed exist.

\begin{Theorem}
    For any sentence $\varphi$ such that $\msf{BST}{+}\{\neg\varphi,\alpha\}$ is consistent for some $\msf{UC}$-instance $\alpha$, there is a non-trivial $\msf{NST}$-fragment that becomes trivial upon asserting $\varphi$.
\end{Theorem}
\begin{proof}
    Let $\alpha:=\exists y\forall x(x\in y \lra \psi(x))$.
    Consider the following $\msf{UC}$-instance $\beta$:
    \[
    \beta:=\exists y\forall x(x\in y\lra ((\neg\varphi\ra\psi(x))\land(\varphi\ra x\notin x)))
    \]
    Then $\msf{BST}{+}\beta$ is consistent, as it is equivalent to $\msf{BST}{+}\{\neg\varphi,\alpha\}$. However, $\msf{BST}{+}\{\beta,\varphi\}$ is inconsistent because, if $\varphi$ holds, the set defined by $\beta$ becomes equivalent to the Russell set. 
\end{proof}

\begin{Corollary}
    Let $\alpha$ be the $\msf{UC}$-instance asserting the existence of the co-Russell set\textup{:}
    \[
    \alpha:=\exists y\forall x(x\in y \lra x\in x)
    \]
    Let $\varphi$ be the sentence asserting the self-membership of the co-Russell set\textup{:}
    \[
    \varphi:=\forall y(\forall x(x\in y \lra x\in x)\ra y\in y)
    \]
    Define $\beta$ as follows\textup{:}
    \[
    \beta:=\exists y\forall x(x\in y\lra ((\varphi\ra x\in x)\land(\neg\varphi\ra x\notin x)))
    \]
    Then $\msf{BST} {+} \beta$ is consistent, and it is provable within this theory that the co-Russell set exists and is a member of itself.
\end{Corollary}

Lastly, since essentially any set theory can be viewed as an $\msf{NST}$-fragment, we can ask how the self-membership of the co-Russell set behaves in various known set theories.

\begin{Theorem}
    The following statements hold\textup{:}\footnote{$\msf{ZF}$ stands for Zermelo-Fraenkel set theory, $\msf{ZF{-}FA}$ stands for Zermelo-Fraenkel set theory without the Foundation Axiom, $\msf{AFA}$ stands for the Anti-Foundation Axiom as defined in \citep{aczel}, $\msf{BAFA}$ stands for the Boffa's Anti-Foundation Axiom, $\msf{NF}$ stands for New Foundations, and $\msf{GPK^+_\infty}$ stands for a particular \textit{positive} set theory, as defined in \citep{esser}.}
    \begin{itemize}
        \item[\textup{(a)}]In $\msf{ZF}$, the co-Russell set does not belong to itself.
        \item[\textup{(b)}]In $\msf{ZF{-}FA{+}BAFA}$, the co-Russell set does not exist.
        \item[\textup{(c)}]In $\msf{ZF{-}FA{+}AFA}$, the co-Russell set does not exist.
        \item[\textup{(d)}]In $\msf{ZF{-}FA}$, the self-membership of the co-Russell set is independent and there exist extensions in which the co-Russell set is non-empty and does not contain itself.
        \item[\textup{(e)}]In $\msf{NF}$, the co-Russell set does not exist.
        \item[\textup{(f)}]In $\msf{GPK^+_\infty}$, the co-Russell set belongs to itself.
    \end{itemize}
\end{Theorem}
\begin{proof}

    (a) By the Foundation Axiom, no set can contain itself. Therefore, the co-Russell set must be empty.

    (b) In $\msf{ZF{-}FA{+}BAFA}$, there exists a proper class of Quine atoms. Therefore, the co-Russell set is also a proper class in this context.

    (c) In $\msf{ZF{-}FA{+}AFA}$, there exists a non-well-founded set $A=\{x:x=A\lor x=B\}$ for every well-founded set $B$. Therefore, the co-Russell set is a proper class in this context.

    (d) Consider a cumulative hierarchy $V'$ on a single Quine atom. $V'$ satisfies $\msf{ZF{-}FA}$ and contains a single self-containing set, the Quine atom. This Quine atom is equal to the co-Russell set. On the other hand, consider a cumulative hierarchy $V''$ on two Quine atoms. $V''$ satisfies $\msf{ZF{-}FA}$ and contains exactly two self-containing sets, the Quine atoms. A set containing exactly these two Quine atoms is equal to the co-Russell set.

    (e) $\msf{NF}$ is closed under complements. If the co-Russell set existed, the Russell set would also have to exist, leading to a contradiction.

    (f) $\msf{GPK^+_\infty}$ includes a positive comprehension schema, meaning $\msf{UC}$ is restricted to positive formulas (those without negation or implication). We defer the proof to Section \ref{paradox}, where we demonstrate the existence of the co-Russell set $H^+$ using only positive instances of $\msf{UC}$. It will be evident that $H^+ \in H^+$.
\end{proof}

Thus, it appears there are various mathematical determinants influencing the self-membership of the co-Russell set. In Section \ref{paradox}, we identify additional determinants within very weak extensions of $\msf{BST}$.

In our brief investigation, we found no property tied to independence that distinguishes 
the self-membership assertion of the co-Russell set
from other set-theoretic sentences.\footnote{Readers can compare this situation with, for example, the results in \citep{honzik}, which show that nearly all of the large cardinals defined so far do not settle the status of the continuum hypothesis.} This leads us to conclude that the concept of a set-theoretic hypodox remains unclear in the formal context—at least for now. Specifically, if the notion is not to be redundant and the self-membership assertion of the co-Russell set should qualify as a hypodox, while other set-theoretic sentences with no apparent connection to paradoxes should not.

\section{co-Russell's paradox}\label{paradox}
In this section, we derive a contradiction in an $\msf{NST}$-fragment based on the unique properties of the co-Russell set.
Specifically, we define the sets $H^+$ and $H^-$, which, by $\msf{exten}$, will both be equal to the co-Russell set $\{x : x \in x\}$, but such that $H^+ \in H^+$ and $H^- \notin H^-$. This distinction contradicts the axioms for equality.

The assumptions required to prove the existence of $H^+$ and $H^-$ correspond to the determinants of the self-membership assertion of the co-Russell set referenced at the end of the previous section.

We rely on the following four instances of $\msf{UC}$:\footnote{The form of these constructions was adapted from James E. Hanson’s proof of the self-membership of the co-Russell set in $\msf{GPK_\infty^+}$, shared on \textit{Mathematics Stack Exchange} (\href{https://math.stackexchange.com/q/4515237}{math.stackexchange.com/q/4515237}). Hanson’s approach, in turn, draws on methods from \citep{cantini}.}
\begin{equation}
    \text{ pairing }(\msf{pair})\textup{:}\quad\exists y\forall x(x\in y \leftrightarrow (x=a\lor x=b)) 
\end{equation}

For any two variables/constants $a,b$ and formula $\varphi$, let us define the following abbreviations:\footnote{When using abbreviations, we assume that variables are chosen so that there are no conflicts in the context of use.}
\[
\begin{aligned}
    \varphi(\{a,b\}) :=&\quad\exists y(\varphi(y)\land\forall x(x\in y \leftrightarrow (x=a\lor x=b)))\\
    \varphi(\tuple{a,b}):=&\quad \exists y(\varphi(y)\land\forall x(x\in y \leftrightarrow (x=\{a,a\}\lor x=\{a,b\})))
\end{aligned}
\]
\begin{equation}
    \text{extracting }(\msf{extract})\textup{:}\quad\exists y \forall x(x\in y \leftrightarrow \exists z(z=\tuple{a,x} \land z \in a))
\end{equation}
For any variable/constant $a$ and formula $\varphi$, let us define the following abbreviation:
\[
    \varphi(a[a]):=\quad\exists y(\varphi(y)\land\forall x(x\in y \leftrightarrow \exists z( z=\tuple{a,x}\land z\in a)))
\]
\begin{equation}
    \exists y\forall x(x\in y \leftrightarrow \exists v,w (x{=}\tuple{v,w}\land (w\in w \lor w{=}v[v])))
\end{equation}
Let $A$ be a constant s.t. $\forall x(x\in A \leftrightarrow \exists v,w (x{=}\tuple{v,w}\land (w\in w \lor w{=}v[v])))$.
\begin{equation}
    \exists y\forall x(x\in y \leftrightarrow \exists v,w(x{=}\tuple{v,w}\land (w\in w \land w{\neq}v[v])))
\end{equation}
Let $B$ be a constant s.t. $\forall x(x\in B \leftrightarrow \exists v,w(x{=}\tuple{v,w}\land (w\in w \land w{\neq}v[v])))$.

Now, we derive a contradiction from $\msf{BST}$ and the above instances of $\msf{UC}$. Let $H^+$ be a constant such that $H^+{=}A[A]$ and $H^-$ be a constant such that $H^-{=}B[B]$. $H^+$ and $H^-$ exist by $\msf{extract}$.
\begin{Lemma}\label{theorem1}
    The following holds for all $x$\textup{:}
\[
\begin{aligned}
    x\in H^+\ \longleftrightarrow\ (x\in x \lor x{=}H^+)\\
    x\in H^-\ \longleftrightarrow\ (x\in x \land x{\neq}H^-)
\end{aligned}
\]
\end{Lemma}
\begin{proof}
    By simply unpacking the definitions, we obtain:
    \[
    \begin{aligned}
        x\in H^+\ &\longleftrightarrow\ \exists z( z{=}\tuple{A,x} \land z\in A )\ &\longleftrightarrow\ \exists z( z{=}\tuple{A,x} \land (x\in x \lor x{=}H^+) ) \\
        x\in H^-\ &\longleftrightarrow\ \exists z( z{=}\tuple{B,x} \land z\in B )\ &\longleftrightarrow\ \exists z( z{=}\tuple{B,x} \land (x\in x \land x{\neq}H^-) )
    \end{aligned}
    \]
In the rightmost equivalence we have replaced $A[A]$ ($B[B]$) with $H^+$ ($H^-$) because $H^+{=}A[A]$ ($H^-{=}B[B]$). This substitution is the core idea.

By $\msf{pair}$, we know that if $y$ and $x$ exist, then $\exists z(z=\tuple{y,x})$ holds. So any sentence of the form $\exists z(z=\tuple{y,x} \land \varphi)$ is equal to $\varphi$. Last sentences in the chains of equivalences above are of such form.
\end{proof}

\begin{Theorem}\label{theorem2}
    $H^+{\in}H^+$
\end{Theorem}
\begin{proof}
From Lemma \ref{theorem1}, we obtain the following.
\[
H^+{\in} H^+\ \longleftrightarrow\ (H^+{\in}H^+ \lor H^+{=}H^+)
\]
The right side of equivalence is clearly true, since $H^+{=}H^+$ holds. Thus $H^+{\in}H^+$.
\end{proof}

\begin{Theorem}\label{theorem3}
    $H^-{\notin}H^-$
\end{Theorem}
\begin{proof}
    Assume that $H^-{\in}H^-$. From Lemma \ref{theorem1}, we obtain the following.
\[
H^-{\in}H^-\ \longleftrightarrow\ (H^-{\in}H^- \land H^-{\neq}H^-)
\]
The right side of equivalence is clearly false, since $H^-{\neq}H^-$ is false. Thus $H^-{\notin}H^-$.
\end{proof}

\begin{Corollary}\label{cl1}
    $H^+{\neq}H^-$
\end{Corollary}
\begin{proof}
    Follows from Theorem \ref{theorem2}, Theorem \ref{theorem3} and axioms of equality.
\end{proof}

\begin{Lemma}\label{theorem4}
    $H^-{\subseteq}H^+$
\end{Lemma}
\begin{proof}
    We need to prove that $\forall x(x\in H^- \rightarrow x\in H^+)$. By substituting for the equal statements from Lemma \ref{theorem1}, we get:
    \[
    \forall x((x\in x \land x{\neq}H^- )\rightarrow(x\in x \lor x{=}H^+ ) )
    \]
    Let us pick a new constant $c$ and assume that the following (antecedent for $c$) holds: $(c\in c \land c{\neq}H^-)$.
    By assumption, $c\in c$ holds, and therefore the following (succedent for $c$) holds: $(c\in c \lor c{=}H^+)$.
    Thus, $H^-{\subseteq}H^+$.
\end{proof}

\begin{Lemma}\label{theorem5}
    $H^-{\supseteq}H^+$
\end{Lemma}
\begin{proof}
    We need to prove that $\forall x(x\in H^+ \rightarrow x\in H^-)$. By substituting for the equal statements from Lemma \ref{theorem1}, we get:
    \[
    \forall x((x\in x \lor x{=}H^+) \rightarrow (x\in x \land x{\neq}H^- ) )
    \]
    Let us pick a new constant $c$ and assume that the following (antecedent for $c$) holds: $(c\in c \lor c{=}H^+)$.
    By assumption, $(c\in c\lor c{=}H^+)$ holds, but since $H^+{\in}H^+$ by Theorem \ref{theorem2}, we can conlude that $c\in c$ holds. Therefore, it must be that $c{\neq}H^-$, since $H^-{\notin}H^-$ by Theorem \ref{theorem3}. Thus, the following (succedent for $c$) holds: $(c\in c \land c{\neq}H^-)$.
    Thus, $H^-{\supseteq}H^+$.
\end{proof}

\begin{Corollary}\label{cl2}
    $H^+{=}H^-$
\end{Corollary}
\begin{proof}
    Follows from Lemma \ref{theorem4}, Lemma \ref{theorem5} and $\msf{exten}$.
\end{proof}

\begin{Corollary}
    $\bot$
\end{Corollary}
\begin{proof}
    Follows from Corollary \ref{cl1} and Corollary \ref{cl2}.
\end{proof}

\section{Fixed Point Theorem and ultimate diagonal set}
The core construction from the previous section plays a role in proving a variant of the so-called Fixed Point Theorem.\footnote{The origins of this theorem can be traced to \cite[p. 30]{girard}.} This variant essentially demonstrates that the restriction in $\msf{UC}$ requiring $y$ not to be free in $\varphi$ can be circumvented within sufficiently strong (but not necessarily inconsistent) $\msf{NST}$-fragments.

We include a proof of this variant in this paper because, in other presentations\footnote{See, e.g., \cite[p. 382]{petersen} and \cite[p. 174]{weberbook}.}, this version is less explicit, and the proof either employs more advanced logical symbolism or formally relies on set-builder notation. Moreover, we consider this variant to be significant for hypodox research.

We use the following schema of instances of $\msf{UC}$:
\[
    \varphi{\text{-separation }} (\varphi{\text{-}}\msf{separ})\textup{:}\quad\exists y\forall x(x\in y \leftrightarrow \exists v\exists w (x{=}\tuple{v,w}\land \varphi(w,v[v]))) 
\]
Let $A_\varphi$ be a constant such that $\forall x(x\in A_\varphi \leftrightarrow \exists v,w (x{=}\tuple{v,w}\land \varphi(w,v[v])))$.

\begin{Theorem}
    For any formula $\varphi(x,y)$, 
    there exists a formula $\psi(x)$ such that $y$ is not free in $\psi$, and the following equivalence holds in $\msf{BST{+}\{pair,extract,\varphi{\text{-}}\msf{separ}\}}{:}$
    \[
    \exists y \forall x(x\in y\lra \varphi(x,y))\lra \exists y \forall x(x\in y \lra \psi(x)).
    \]
\end{Theorem}

\begin{proof}
    By $\varphi{\text{-}}\msf{separ}$, the following set $A_\varphi$ exists:
    \[
    \forall x(x\in A_\varphi \leftrightarrow \exists v\exists w (x{=}\tuple{v,w}\land \varphi(w,v[v])))
    \]
    Let $A_\varphi^2=A_\varphi[A_\varphi]$ by $\msf{extract}$. 
    By unpacking the definitions, we obtain:
    \[
    x\in A_\varphi^2\ \longleftrightarrow\ \exists z( z{=}\tuple{A_\varphi,x} \land z\in A_\varphi )\ \longleftrightarrow\ \exists z( z{=}\tuple{A_\varphi,x} \land \varphi(x,A_\varphi^2) )
    \]
In the rightmost equivalence, we have replaced $A_\varphi[A_\varphi]$ with $A_\varphi^2$ because $A_\varphi^2{=}A_\varphi[A_\varphi]$. This substitution is the core idea.

By $\msf{pair}$, we know that if $y$ and $x$ exist, then $\exists z(z=\tuple{y,x})$ holds. Thus, any sentence of the form $\exists z(z=\tuple{y,x} \land \varphi)$ is equivalent to $\varphi$. The last sentence in the chain of equivalences above is of this form.
Therefore, we conclude that for all $x$:
    \[
    x\in A_\varphi^2\ \longleftrightarrow\ \varphi(x,A_\varphi^2) 
    \]
    We take $x\in A_\varphi^2$ as $\psi$.
\end{proof}

The Fixed Point Theorem can be used to derive contradictions, which have not been previously regarded as self-standing paradoxes. Consider the following set $Z$:
\[
\forall x(x\in Z \leftrightarrow \exists v\exists w (x{=}\tuple{v,w}\land  w\notin v[v]))
\]
Let us call the sentence above $Z{\text{-}}\msf{separ}$.
\begin{Theorem}
    $\msf{BST{+}\{pair,extract},Z{\text{-}}\msf{separ}\}$ is inconsistent.
\end{Theorem}
\begin{proof}
    Let $Z^2=Z[Z]$ by $\msf{extract}$. By unpacking the definitions and using $\msf{pair}$ (as in the proof the Fixed Point Theorem), we find that for all $x$:
\[
x\in Z^2\ \longleftrightarrow\  x \notin Z^2 
\]
Thus, $Z^2$ contains exactly those sets that $Z^2$ does not contain. This is a limiting case\footnote{There are contexts in which sets like $Z^2$ can be useful; see, e.g., \citep{weber}.} of diagonal definitions, and the contradiction follows immediately.
\end{proof}

\section{Final remarks}

In this paper, we explored the concept of a hypodox in the hope that it could serve as a useful tool in formal contexts. We found that, although the concept of a set-theoretic paradox can be relatively well-defined in formal settings, the concept of a~set-theoretic hypodox lacks similar clarity, at least for now.

However, we did uncover a distinctive property of the co-Russell set—the primary example of a hypodoxical set—that sets it apart from other sets. Specifically, we demonstrated that a contradiction can be derived through the Fixed Point Theorem in $\msf{NST}$, leveraging the unique properties of the co-Russell set.

Could this property serve as the basis for characterizing set-theoretic hypodoxes? The answer depends on the criteria we impose on such a collection. If we use this property as a defining feature, it would exclude at least some complements of paradoxical sets other than the Russell set. 

Consider, for instance, the paradoxical Curry-like set from the proof of Theorem \ref{thmpara}. Its complement would be the set $D = \{x : x\in x \land \psi\}$, where $\psi$ asserts the existence of the universal set. In $\msf{BST}$, the self-membership of $D$ is independent, similar to the co-Russell set. However, unlike the co-Russell set, there is no issue with having the following two sets within the same model:
\[
\begin{aligned}
    D^+ :=&\quad \{x:(x\in x \land \psi)\lor x=D^+\}\\
    D^- :=&\quad \{x:(x\in x \land \psi)\land x\neq D^-\}
\end{aligned}
\]
If there is no universal set in the model (and thus $\psi$ does not hold), then $D^-$ collapses to $\emptyset$, and $D^+$ collapses to a set containing only $D^+$.

This suggests that either the correct distinguishing property for set-theoretic hypodoxes has yet to be identified, or the paradox-hypodox duality begins to break down in certain cases.

\bibliography{bibliography}

\end{document}

%% file: base.tex
\usepackage[utf8]{inputenc}
\usepackage[a4paper,
            bindingoffset=0.2in,
            left=1.5in,
            right=1.5in,
            top=1in,
            bottom=1in,
            footskip=.25in]{geometry}
\usepackage[
  main=english, 
  english
]{babel} 
\usepackage[T1,T2A]{fontenc}  
\usepackage[authoryear]{natbib}
\usepackage{doi}

\usepackage{csquotes}
\AtBeginEnvironment{displayquote}{\smaller[0.5]}

\usepackage{hyperref}

\usepackage{verbatim}

\usepackage{enumitem}

\usepackage{relsize,etoolbox}

\usepackage[nottoc]{tocbibind}

\usepackage{amsmath}
\usepackage{amsfonts} 
\usepackage{amssymb}
\usepackage{amsthm}


%% file: logic.tex
\DeclareMathSymbol{\mlq}{\mathord}{operators}{``}
\DeclareMathSymbol{\mrq}{\mathord}{operators}{`'}

\newcommand{\N}{\mathbb{N}}

\newcommand{\tuple}[1]{\langle#1\rangle}

\newcommand{\msf}[1]{\mathsf{#1}}

\newcommand{\ra}{\rightarrow}

\newcommand{\lra}{\leftrightarrow}

